\documentclass[12pt]{amsart}
\usepackage{enumerate}
\usepackage{url}
\usepackage{xspace}

\usepackage[margin=1in]{geometry}
\DeclareMathOperator{\pnt}{\raise 0.5mm \hbox{\large\textbf{.}}}

\newcommand{\note}[2][ ]{}%dummy macro

\newtheorem{theorem}{Theorem}%[section]
\newtheorem{lemma}[theorem]{Lemma}
\newtheorem{proposition}[theorem]{Proposition}

\theoremstyle{definition}

\usepackage[pdfauthor={Huy Tai Ha and Fabrizio Zanello and Erik Stokes},
            pdftitle={Pure O-sequences and matroid h-vectors},
            pdfsubject={commutative algebra},
            pdfkeywords={simplicial complex, matroid, h-vector,
              pure O-sequence},
            pdfproducer={Latex with hyperref},
            pdfcreator={latex->dvips->ps2pdf},
            pdfpagemode=UseNone,
            bookmarksopen=false,
            bookmarksnumbered=true]{hyperref}

\title{Zeilberger's KOH theorem and the strict\\ unimodality of $q$-binomial coefficients}
\author{Fabrizio Zanello} \address{Department of Mathematics\\ MIT\\ Cambridge, MA 02139-4307 {\tiny and} Department of Mathematical  Sciences\\ Michigan Tech\\ Houghton, MI  49931-1295}
\email{zanello@math.mit.edu, zanello@mtu.edu}
\thanks{2010 {\em Mathematics Subject Classification.} Primary: 05A15;
  Secondary: 05A17.\\\indent 
{\em Key words and phrases.} $q$-binomial coefficient; Gaussian polynomial; unimodality.}

\begin{document}
\maketitle
%\linenumbers

\begin{abstract}
A recent nice result due to I. Pak and G. Panova is the strict unimodality of the $q$-binomial coefficients $\binom{a+b}{b}_q$ (see \cite{PP} and also \cite{PP2} for a slightly revised version of their theorem). Since their proof used representation theory and Kronecker coefficients,  the authors also asked for an argument that would employ Zeilberger's KOH theorem. In this note, we give such a proof. Then, as a further application of our method, we also provide a short proof of their conjecture that the difference between consecutive coefficients of $\binom{a+b}{b}_q$ can get arbitrarily large, when we assume that $b$ is fixed and $a$ is large enough.
\end{abstract}

%\section{Introduction}
{\ }\\\indent
A sequence $c_1, c_2,\dots, c_t$ is \emph{unimodal} if it does not increase strictly after a strict decrease. It is \emph{symmetric} if $c_i=c_{t-i}$ for all $i$. The unimodality of the \emph{$q$-binomial coefficient}% (or \emph{Gaussian polynomial})

$$\binom {a+b}{b}_q= \frac{(1-q)(1-q^2)\cdots (1-q^{a+b})}{(1-q)(1-q^2)\cdots (1-q^a)\cdot (1-q)(1-q^2)\cdots (1-q^{b})},$$
which is easily proven to be a symmetric polynomial in $q$, is a classical and highly nontrivial  result in combinatorics. It was first shown in  1878  by J.J. Sylvester, and has since received a number of other interesting proofs (see e.g. \cite{Pr1,St5,Sy}). In particular, a celebrated paper  of K. O'Hara \cite{Oh} provided a combinatorial proof for the unimodality of $\binom {a+b}{b}_q$. O'Hara's argument was subsequently expressed in algebraic terms by D. Zeilberger \cite{Z2} by means of the beautiful \emph{KOH identity}. This identity decomposes $\binom {a+b}{b}_q$ into a finite sum of polynomials with nonnegative integer coefficients, which are all unimodal and symmetric about $ab/2$. 

More precisely, fix  integers $a\ge b\ge 2$. For any given partition $\lambda=(\lambda_1,\lambda_2,\dots)$ of $b$, set $Y_i= \sum_{j=1}^i\lambda_j$ for all $i\geq 1$, and $Y_0=0$. Then the KOH theorem can be stated as follows:

\begin{lemma}[KOH]\label{koh} $\binom{a+b}{b}_q=\sum_{\lambda\vdash b}F_{\lambda}(q)$, where
$$F_{\lambda}(q)= q^{2\sum_{i\geq 1}\binom{\lambda_i}{2}} \prod_{j\geq 1} \binom{j(a+2)-Y_{j-1}-Y_{j+1}}{\lambda_j-\lambda_{j+1}}_q.$$
\end{lemma}

A recent nice result shown by I. Pak and G. Panova is a characterization of the \emph{strict} unimodality of  $q$-binomial coefficients; i.e., they determined when $\binom {a+b}{b}_q$ \emph{strictly} increases from degree 1 to degree $\lfloor ab/2\rfloor $ (see \cite{PP}, and also \cite{PP2} for a slightly revised version of the theorem). Since their argument employed the algebraic machinery of Kronecker coefficients, the authors asked whether a proof could also be given that uses Zeilberger's KOH identity. We do this in the present note. Then, as a further pithy application of this method, using the KOH theorem we also give a very short proof of a conjecture stated in the same papers, on the unbounded growth of the difference between consecutive coefficients of $\binom {a+b}{b}_q$.

The next lemma is a trivial and probably well-known fact of which we omit the proof.
\begin{lemma}\label{111}
Let $c$ and $d$ be positive integers such that the $q$-binomial coefficient $\binom{c+d}{d}_q$ is strictly unimodal. Then, for any positive integer $t\le cd$ such that $t\ne cd-2$, the product $\binom{c+d}{d}_q\binom{t+1}{1}_q$ is strictly unimodal (in all nonnegative degrees).
\end{lemma}

\begin{theorem}[\cite{PP,PP2}]\label{str}
The $q$-binomial coefficient $\binom{a+b}{b}_q$ is strictly unimodal if and only if $a=b=2$ or $b\ge 5$, with the exception of$$(a,b)=(6,5), (10,5), (14,5), (6,6), (7,6), (9,6), (11,6), (13,6), (10,7).$$
\end{theorem}

\begin{proof}
We can assume that $b\ge 5$, otherwise, as it is also noted in \cite{PP,PP2}, the result is easy to show. By Lemma \ref{koh}, since all  terms in the KOH decomposition of $\binom{a+b}{b}_q$ are unimodal and symmetric with respect to $ab/2$, in order to show that $\binom{a+b}{b}_q$ is \emph{strictly} unimodal, it clearly suffices to determine, for each positive degree up to $ab/2$, some suitable KOH term that is strictly increasing in that degree. We begin by showing that, for any $a\ge b\ge 2$, $\binom{a+b}{b}_q$  strictly increases up to degree $ab/2 -a$ for $b$  even, and up to degree $ab/2 -a/2$ for $b$  odd. 

Let $b=2m$ be even. Then the KOH term contributed by the partition $\lambda=(\lambda_1=2,\dots,\lambda_{m-1}=2,\lambda_m=1, \lambda_{m+1}=1)$ of $b$ is given by:

$$F_{\lambda}(q)=q^{2(m-1)}\binom{(m-1)(a+2)-2(m-2)-(2m-1)}{1}_q\binom{(m+1)(a+2)-(2m-1)-2m}{1}_q$$$$=q^{b-2}\binom{ab/2-a-b+3}{1}_q\binom{ab/2+a-b+3}{1}_q.$$

Notice that the product of the last two $q$-binomial coefficients is  strictly increasing (by 1) from degree 0 to degree $ab/2-a-b+2$. Also, $\binom{a+b}{b}_q$ is clearly strictly increasing from degree 1 to degree $b-2$, since so is the usual partition function $p(n)$ (see e.g. \cite{ec1}, Chapter 1). From this, we easily have that $\binom{a+b}{b}_q$   strictly increases from degree 1 to degree $(ab/2-a-b+2)+(b-2)=ab/2-a$. 

The proof for $b=2m+1$ odd, giving us that $\binom{a+b}{b}_q$ is  strictly increasing  up to degree $ab/2 -a/2$, is  similar (using $\lambda=(\lambda_1=2,\dots,\lambda_{m}=2,\lambda_{m+1}=1)$) and thus will be omitted.

Now, in order to show that for the desired values of $a$ and $b$, $\binom{a+b}{b}_q$  strictly increases in each of the remaining degrees up to $ab/2$, we consider three cases depending on the residue of $b$ modulo 3. We start with $b\equiv 0$ modulo 3, and assume that $b\ge 15$. The KOH term corresponding to the partition $\lambda=(b/3,b/3, b/3)$ of $b$ is given by:

$$F_{\lambda}(q)=q^{6\binom{b/3}{2}}\binom{3(a+2)-2b/3-b}{b/3}_q=q^{b(b-3)/3}\binom{(3a-2b+6)+b/3}{b/3}_q.$$

Notice that $b(b-3)/3<ab/2-a$, and $3a-2b+6\ge 15$. Thus, it easily follows by induction that, for $b\ge 15$, the strict unimodality of $\binom{(3a-2b+6)+b/3}{b/3}_q$ implies that of $\binom{a+b}{b}_q$, as desired.

Let now $b\equiv 1$ modulo 3, and assume $b\ge 19$. By considering the partition $\lambda=((b-1)/3,(b-1)/3, (b-1)/3,1)$ of $b$, we get:

$$F_{\lambda}(q)=q^{(b-1)(b-4)/3}\binom{3a-2b+8+(b-4)/3}{(b-4)/3}_q\binom{4a-2b+9}{1}_q.$$

It is easy to check that, under the current assumptions on $a$ and $b$, we have $(b-1)(b-4)/3<ab/2-a$ and $(3a-2b+8)(b-4)/3\geq (4a-2b+8)+3$. In particular, we are under the hypotheses of Lemma \ref{111}. Thus, since $3a-2b+8\ge 15$, the strict unimodality of $\binom{a+b}{b}_q$ follows by induction from that of $\binom{3a-2b+8+(b-4)/3}{(b-4)/3}_q$, for all $b\ge 19$, as we wanted to show.

The treatment of the case $b\equiv 2$ modulo 3, $b\ge 20$, is analogous so we will omit the details. We only remark here that one considers the partition $\lambda=((b-2)/3,(b-2)/3, (b-2)/3,1,1)$ of $b$, whose contribution to the KOH expansion of  $\binom{a+b}{b}_q$ is:

$$F_{\lambda}(q)=q^{(b-2)(b-5)/3}\binom{3a-2b+10+(b-5)/3}{(b-5)/3}_q\binom{5a-2b+11}{1}_q.$$

The strict unimodality of $\binom{a+b}{b}_q$, for  $b\ge 20$,  then follows in a similar fashion from that of $\binom{3a-2b+10+(b-5)/3}{(b-5)/3}_q$, by  employing Lemma \ref{111} and induction.

Therefore, it remains to show the theorem for $5\le b\le 17$ ($b\neq 15$). We will assume for simplicity that $a\ge 2b+13$, the result being easy to verify directly for the remaining values of $a$. The KOH term contributed by the partition $(b)$ of $b$ in the expansion of $\binom{a+b}{b}_q$ is:

$$F_{(b)}(q)=q^{2\binom{b}{2}}\binom{a+2-b}{b}_q=q^{b(b-1)}\binom{(a-2b+2)+b}{b}_q.$$

Clearly, since $a\ge 2b+13$, we have $b(b-1)< ab/2-a$ and $a-2b+2\ge 15$. Thus, by induction, the strict unimodality of $\binom{(a-2b+2)+b}{b}_q$  implies that of $\binom{a+b}{b}_q$, as desired.
\end{proof}

In Remark 3.6 of \cite{PP,PP2}, the authors also conjectured that, roughly speaking, the difference between  consecutive coefficients of a $q$-binomial coefficient is eventually larger than any fixed integer. As a further nice, and very brief, application of our method, we answer this conjecture in the positive using the KOH identity. (Just notice that unlike in the original formulation of the conjecture, our proof will assume that $b$ is fixed and only $a$ is large enough.)

\begin{proposition} Fix any integer $d\ge 1$. Then there exist integers $a_0,b$ and $L$ such that, if $\binom{a+b}{b}_q=\sum_{i=0}^{ab}c_iq^i$, then $c_i-c_{i-1}\ge d$, for all indices $L\le i\le ab/2$ and for all $a\ge a_0$.
\end{proposition}

\begin{proof}
Consider the partition $\lambda^{[k]}=(\lambda_1^{[k]}=b-k,\lambda_2^{[k]}=1,\dots, \lambda_{k+1}^{[k]}=1)$ of $b$, where $k\ge 1$. It is easy to see that its contribution to the KOH identity for $\binom{a+b}{b}_q$ is given by:
$$F_{\lambda^{[k]}}(q)=q^{(b-k)(b-k-1)}\binom{a-2b+2k+2+(b-k-1)}{b-k-1}_q\binom{(k+1)(a+2)-2b+1}{1}_q.$$

Set for instance $b=2d+4$ and  $a_0=(d+2)(d+3)+6$, where we can assume $d\ge 2$. A standard computation gives  that, for any $a\ge a_0$ and  $k\le b/2-2=d$,  we are under the hypotheses of Lemma \ref{111}. Hence, by Theorem \ref{str}, each polynomial $F_{\lambda^{[k]}}(q)$ is strictly unimodal from degree $(b-k)(b-k-1)$ on, and the theorem now immediately follows  by choosing  $L=(b-1)(b-2)+1=4d^2+10d+7$ and considering the coefficients of $\sum_{k=1}^dF_{\lambda^{[k]}}(q)$.
\end{proof}

\section{Acknowledgements} The idea of this paper originated during a visit to UCLA in October 2013, for whose invitation we warmly thank Igor Pak. We wish to thank the referee for a careful reading of our manuscript and for helpful suggestions that improved the presentation, and Igor Pak, Greta Panova, and Richard Stanley for several helpful discussions. We are also very grateful to Doron Zeilberger for comments (and for calling this a ``proof from the book''). This work was done while the author was partially supported by a Simons Foundation grant (\#274577).

\end{document}